\documentclass{amsart}

\usepackage{amsmath, amsthm, amssymb}
\usepackage{amsfonts}
\usepackage[ansinew]{inputenc}
\usepackage[dvips]{epsfig}
\usepackage{graphicx}
\usepackage[english]{babel}
\usepackage{pdfsync}
\usepackage{hyperref}
\usepackage{color}

\theoremstyle{plain}
\newtheorem{thm}{Theorem}[section]
\newtheorem{Corollary}[thm]{Corollary}
\newtheorem{Lemma}[thm]{Lemma}
\newtheorem{Proposition}[thm]{Proposition}

\theoremstyle{definition}

\theoremstyle{remark}
\newtheorem{Remark}[thm]{Remark}

\numberwithin{equation}{section}

\newcommand{\average}{{\mathchoice {\kern1ex\vcenter{\hrule height.4pt
width 6pt depth0pt} \kern-9.7pt} {\kern1ex\vcenter{\hrule
height.4pt width 4.3pt depth0pt} \kern-7pt} {} {} }}

\begin{document}

\title[Nonlocal Dirichlet problems with $C^{0,\alpha}$ exterior data]{The Dirichlet problem for nonlocal elliptic operators with $C^{0,\alpha}$ exterior data}

\author{Alessandro Audrito}
\address{Politecnico di Torino, Institute of Mathematics DISMA, Corso Duca degli Abruzzi 24, 10129, Torino, Italy \newline \indent Universit\"at Z\"urich, Department of Mathematics, Winterthurerstrasse 190, CH-8057 Z\"urich, Switzerland}
\email{alessandro.audrito@polito.it; alessandro.audrito@math.uzh.ch}

\author{Xavier Ros-Oton}
\address{Universit\"at Z\"urich, Institute of Mathematics, Winterthurerstrasse 190, CH-8057 Z\"urich, Switzerland  \newline \indent ICREA, Passeig Llu\'is Companys 23, 08010 Barcelona, Spain
 \newline \indent 
Universitat de Barcelona, Departament de Matem\`atiques i Inform\`tica, Gran Via 585, 08007 Barcelona, Spain}
\email{xavier.ros-oton@math.uzh.ch}


\thanks{A.A. was supported by PRIN 2015 (MIUR, Italy), and by the GNAMPA projects ``Esistenza e propriet\`a qualitative per soluzioni di EDP non lineari ellittiche e paraboliche'' and ``Ottimizzazione Geometrica e Spettrale'' (Italy). X.R. was supported by the European Research Council under the Grant Agreement No. 801867 ``Regularity and singularities in elliptic PDE (EllipticPDE)'', by the Swiss National Science Foundation, and by MINECO grant MTM2017-84214-C2-1-P}

\begin{abstract}
In this note we study the boundary regularity of solutions to nonlocal Dirichlet problems of the form $Lu=0$ in $\Omega$, $u=g$ in $\mathbb R^N\setminus\Omega$, in non-smooth domains $\Omega$.
When $g$ is smooth enough, then it is easy to transform this problem into an homogeneous Dirichlet problem with a bounded right hand side, for which the boundary regularity is well understood.
Here, we study the case in which $g\in C^{0,\alpha}$, and establish the optimal H\"older regularity of $u$ up to the boundary.
Our results extend previous results of Grubb for $C^\infty$ domains~$\Omega$.
\end{abstract}

\maketitle

\section{Introduction} \label{SectionIntroduction}

Given a bounded Lipschitz domain $\Omega \subset \mathbb{R}^N$, we study the regularity of solutions to nonlocal Dirichlet problems of the form
\begin{equation}\label{eq:NonlocalDirichletProblem}
\left\{\begin{array}{rcll}
L u &=& 0 \quad &\text{in } \Omega \\
u& =& g             \quad &\text{in } \mathbb{R}^N\setminus\Omega,
\end{array}\right.
\end{equation}
where $L$ is an operator of the form
\begin{equation}\label{eq:NonlocalNonHomogenousOperator}
-Lu(x) = \int_{\mathbb{R}^N} \left( u(x) - \frac{u(x+y) + u(x-y)}{2} \right) K(y) dy,
\end{equation}
with kernel $K$ satisfying
\begin{equation}\label{eq:KernelProp}
 K(y) = K(-y) \quad \textrm{and} \quad
 \frac{\lambda}{|y|^{N+2s}} \leq K(y) \leq \frac{\Lambda}{|y|^{N+2s}}, \quad y  \in \mathbb{R}^N \setminus \{0\}.
\end{equation}
Here, $s \in (0,1)$ and $0 < \lambda \leq \Lambda$.

In most of our results we will assume in addition that
\begin{equation}\label{eq:NonlocalHomogenousOperator}
K \ \textrm{is homogeneous.}
\end{equation}
Notice that, when  $\lambda = \Lambda$, we recover (a multiple of) the fractional Laplacian $(-\Delta)^s$.
Even in that case, the results that we establish in this paper were only known for $C^\infty$ domains~$\Omega$; see Grubb \cite[Theorem 2.5]{Grubb2014}.

The existence, regularity, and further properties of solutions to nonlocal Dirichlet problems of this type has been an active topic of research in the last years; we refer to \cite{BarlesImbert2008:art, Bogdan1997:art, BogdanGrzywnyRyznar2015:art,CS09,  FelKassVoigt2015:art, Mou2017:art, RosOton2016:art} and references therein.

When $g$ is smooth enough (e.g., $g\in C^{2s+\varepsilon}$ for some $\varepsilon>0$) then it is easy to transform \eqref{eq:NonlocalDirichletProblem} into a homogeneous Dirichlet problem of the type $Lu=f$ in $\Omega$, $u=0$ in $\Omega^c$, with $f\in L^\infty(\Omega)$.
It is well known then that (as long as $\Omega$ is smooth enough), solutions $u$ are $C^{0,s}$ up to the boundary.

However, when $g$ is less regular (e.g., $g\in C^{0,\alpha}$) then the boundary regularity of solutions to \eqref{eq:NonlocalDirichletProblem} has to be treated more carefully, and to the best of our knowledge this has only been studied in case of $C^\infty$ domains by \cite{Grubb2014}.

Roughly speaking, the aim of this paper is to show that, when $g\in C^{0,\alpha}$ ($\alpha>0$ small) and $\Omega$ is Lipschitz, then $u$ is $C^{0,\alpha}$ up to the boundary.
This is explained in more detail next.

\subsection{The local case}
When $s=1$, solutions to the Dirichlet problem
\[
\left\{\begin{array}{rcll}
\Delta u &=& 0 \quad &\text{in } \Omega \\
u& =& g             \quad &\text{on } \partial\Omega
\end{array}\right.
\]
satisfy the following (cfr. \cite{GilbargTrudinger2001:book,Kenig:book}):
\begin{itemize}

\item[(a)] If $g \in C^{0,\alpha}(\partial\Omega)$ for some $\alpha \in (0,1)$, and $\Omega$ is at least $C^1$, then $u \in C^{0,\alpha}(\overline{\Omega})$.

\item[(b)] If $g \in C^{0,1}(\partial\Omega)$ then in general $u \not\in C^{0,1}(\overline{\Omega})$, even if $\Omega$ is of class $C^{\infty}$.

\item[(c)] If $g \in C^{1,\alpha}(\partial\Omega)$ for some $\alpha \in (0,1)$, and $\Omega$ is at least $C^{1,\alpha}$, then $u \in C^{1,\alpha}(\overline{\Omega})$.
\end{itemize}

Finally, when $\Omega$ is not $C^1$ but only Lipschitz, we have the following:
\begin{itemize}

\item[(d)] If $\Omega$ is Lipschitz, then there exists $\alpha_0=\alpha_0(\Omega)>0$ such that if $g \in C^{0,\alpha}(\partial\Omega)$ for some $\alpha \in (0,\alpha_0]$, then $u \in C^{0,\alpha}(\overline{\Omega})$.

\end{itemize}
The above results are sharp in terms of the regularity of $g$, and also in terms of the regularity of $\Omega$.

%
%
%
%
%
%
\subsection{Our results}

The goal of this paper is to provide analogous results to (a), (b), (c), and (d) for nonlocal Dirichlet problems of the type \eqref{eq:NonlocalDirichletProblem}, with $s\in(0,1)$.

The right assumption on the exterior datum $g$ turns out to be
\begin{equation}\label{eq:AssumptionsBoundaryData}
|g(x) - g(z)| \leq C_0 |x - z|^{\alpha} \quad \text{ for all } x \in \mathbb{R}^N\setminus\Omega, \; z \in \partial\Omega,
\end{equation}
for some constant $C_0$ and $\alpha\in(0,1)$.
Notice that, in particular, $g$ is $C^{0,\alpha}$ on $\partial \Omega$ (but not necessarily outside $\overline\Omega$).
Moreover, taking $C_0$ larger if necessary, $g$ will satisfy the growth condition
\begin{equation}\label{eq:GrowthCondition}
|g(x)| \leq C_0(1 + |x|^{\alpha}), \qquad x \in \mathbb{R}^N \setminus \Omega,
\end{equation}

Our first (and main) result provides the analogue of property (a) above.

\begin{thm}[$\alpha<s$]\label{Theorem:MainTheorem}
Let $\Omega\subset\mathbb R^N$ be any bounded $C^1$ domain, $s \in (0,1)$, $L$ as in \eqref{eq:NonlocalNonHomogenousOperator}-\eqref{eq:KernelProp}-\eqref{eq:NonlocalHomogenousOperator} and $g$ as in \eqref{eq:AssumptionsBoundaryData}-\eqref{eq:GrowthCondition}, with $\alpha\in(0,s)$.

Then the solution $u$ to \eqref{eq:NonlocalDirichletProblem} is of class $C^{0,\alpha}(\overline{\Omega})$, with
\[\|u\|_{C^{0,\alpha}(\overline\Omega)} \leq CC_0,\]
where $C$ depends only on $n$, $s$, $\lambda$, $\Lambda$, $\alpha$, and $\Omega$.
\end{thm}

Moreover, we will show that the previous result fails when $\alpha=s$, even if $\Omega$ is smooth.
This is the analogue of property (b) above.

\begin{Proposition}[$\alpha=s$]\label{Proposition:MainProp}
Let $s \in (0,1)$ and $-L=(-\Delta)^s$.

Then, there exists a $C^\infty$ domain $\Omega\subset\mathbb R^2$ and a function $g$ satisfying \eqref{eq:AssumptionsBoundaryData}-\eqref{eq:GrowthCondition} with $\alpha=s$, such that the solution $u$ to \eqref{eq:NonlocalDirichletProblem} satisfies $u\notin C^{0,s}(\overline\Omega)$.
\end{Proposition}

When $\alpha>s$, using known results from \cite{RosOtonSerra2017:art}, we will establish the following.
Notice that, for nonlocal operators of this type, the best H\"older regularity one can get is $C^{0,s}(\overline\Omega)$, even if $g$ and $\Omega$ are $C^\infty$; see \cite{RosOton2016:art}.
This is why the analogue of property (c) above reads as follows.

\begin{Proposition}[$\alpha>s$]\label{Proposition:MainProp2}
Let $\Omega\subset\mathbb R^N$ be any bounded $C^{1,\gamma}$ domain, $\gamma>0$, $s \in (0,1)$, $L$ as in \eqref{eq:NonlocalNonHomogenousOperator}-\eqref{eq:KernelProp}-\eqref{eq:NonlocalHomogenousOperator} and $g$ as in \eqref{eq:AssumptionsBoundaryData}-\eqref{eq:GrowthCondition}, with $\alpha>s$ and $\alpha<2s$.

Then the solution $u$ to \eqref{eq:NonlocalDirichletProblem} is of class $C^{0,s}(\overline{\Omega})$, with
\[\|u\|_{C^{0,s}(\overline\Omega)} \leq CC_0,\]
where $C$ depends only on $n$, $s$, $\lambda$, $\Lambda$, $\alpha$, and $\Omega$.
\end{Proposition}

Finally, when $\Omega$ is of class $C^{0,1}$ we establish the following result, analogue to (d) above.
Notice that here we do not need to assume that the kernel $K$ is homogeneous.

\begin{thm}[$\partial\Omega\in \textrm{Lip}$]\label{Theorem:MainTheoremLipschitz}
Let $\Omega\subset\mathbb R^N$ be any bounded Lipschitz domain, $s \in (0,1)$, and $L$ as in \eqref{eq:NonlocalNonHomogenousOperator}-\eqref{eq:KernelProp}.
Then, there exists $\beta_0>0$, depending only on $\Omega$, $s$, $\lambda$, and $\Lambda$, such that the following holds.
Let $g$ be as in \eqref{eq:AssumptionsBoundaryData}-\eqref{eq:GrowthCondition}, with $\alpha\in (0,\beta_0]$.

Then, the solution $u$ to \eqref{eq:NonlocalDirichletProblem} is of class $C^{0,\alpha}(\overline{\Omega})$, with
\[\|u\|_{C^{0,\alpha}(\overline\Omega)} \leq CC_0,\]
where $C$ depends only on $n$, $s$, $\lambda$, $\Lambda$, $\alpha$, and $\Omega$.
\end{thm}

The strategy in our proof of the $C^{0,\alpha}$ regularity of $u$ is as follows.
The basic idea is to extend the exterior data $g$ to a function $\overline{g}$, defined in $\mathbb R^N$, and such that it is as regular as it can be inside $\Omega$.
Then, we show that $|L\overline{g}| \leq Cd^{\alpha-2s}$ in $\Omega$, where $d(x) := \text{dist}\left(x,\mathbb{R}^N\setminus\Omega \right)$.
Thanks to this, defining $v = u - \overline{g}$, we are led to the study of the problem
\begin{equation}\label{eq:SecondNonlocalDirichletProblem}
\left\{\begin{array}{rcll}
L v &=&f \quad &\text{in } \Omega \\
v& =& 0             \quad &\text{in } \mathbb{R}^N\setminus\Omega,
\end{array}\right.
\end{equation}
with $|f| \leq Cd^{\alpha-2s}$ in $\Omega$.
We then prove regularity properties up to the boundary of solutions $v$ to \eqref{eq:SecondNonlocalDirichletProblem}, and show that
\[
\|v\|_{C^{0,\alpha}(\overline{\Omega})} \leq C \|d^{2s-\alpha} f\|_{L^\infty(\Omega)},
\]
for some $C > 0$. To do this, we need to construct fine barriers, which must take into account two important features: first, $f$ is very singular near the boundary $\partial\Omega$; and second, the domain $\Omega$ is only $C^1$ (or $C^{0,1}$).

\begin{Remark}
In the statements above, we have not specified the nature of the solutions we consider (weak or viscosity). 
Whenever weak or viscosity solutions exist, our results apply to them. Actually, using our new regularity results it is easy to prove the existence of a viscosity solution to the problem \eqref{eq:NonlocalDirichletProblem} under our assumptions on $L$, $\Omega$, and $g$; see Corollary~\ref{cor-existence}.
\end{Remark}

The paper is organized as follows.
In Section \ref{Section:PreliminaryResults} we present some preliminary results we will employ later in the main proofs.
Section \ref{Section:C1Domains} is the core of the paper: we prove Theorem \ref{Theorem:MainTheorem} and Propositions \ref{Proposition:MainProp} and \ref{Proposition:MainProp2}.
Finally, in Section \ref{Section:LipschitzDomains} we consider domains of class $C^{0,1}$, showing Theorem \ref{Theorem:MainTheoremLipschitz}.

%
%
%
%
%
%
%
%
%
%
%
\section{Preliminary results} \label{Section:PreliminaryResults}

This section is devoted to the proof of some preliminary results.
The first one is a $L^{\infty}$ bound on (weak) solutions, based on the maximum principle.
%
%
%
%
%
\begin{Lemma}\label{Lemma:LInfinityBound}
Let $\Omega \subset \mathbb{R}^N$ be any bounded domain, $s \in (0,1)$, $\alpha \in (0,2s)$, $L$ as in \eqref{eq:NonlocalNonHomogenousOperator}-\eqref{eq:KernelProp}, and $g$ satisfying \eqref{eq:GrowthCondition}.

Then, the solution $u$ to \eqref{eq:NonlocalDirichletProblem} is bounded in $\Omega$ and satisfies
\[
\|u\|_{L^{\infty}(\Omega)} \leq CC_0,
\]
with $C$ depending only on $N$, $s$, $\alpha$, $\Omega$, and the ellipticity constants.
\end{Lemma}

\begin{proof}
First notice that, dividing by a constant if necessary, we may assume that $C_0=1$. Since the function $g$ is not assumed to be bounded, the boundedness of $u$ in $\Omega$ does not follow immediately. Thus, we construct an equivalent version of problem \eqref{eq:NonlocalDirichletProblem}, with a nontrivial r.h.s. $f\in L^\infty(\Omega)$ and a new exterior data $g_1 \in L^\infty(\mathbb{R}^N\setminus\Omega)$. Then we will apply some known $L^\infty$ estimates, obtaining the proof of our statement.

Let $R_0 > 0$ be such that $\Omega \subset B_{R_0/2}$, and we consider the functions
\[
g_1(x) := g\, \chi_{B_{R_0} \setminus \Omega}
\qquad \textrm{and} \qquad
g_2 := g\, \chi_{\mathbb{R}^N \setminus B_{R_0}},
\]
so that $g = (g_1+g_2)|_{\mathbb{R}^N \setminus \Omega}$. The function $w := u - g_2$ satisfies
\[
\left\{\begin{array}{rcll}
-Lw &=& f \quad &\text{in } \Omega \\
w& =& g_1             \quad &\text{on } \mathbb{R}^N\setminus\Omega,
\end{array}\right.
\]
with $f := Lg_2$.
Now, since $g_2$ is supported outside $B_{R_0}$, and $B_{R_0}\supset\supset\Omega$, then it is easy to check that $|f|\leq C$ in $\Omega$.
Therefore, since $|g_1|\leq C$ in $\mathbb{R}^N\setminus\Omega$, it follows from \cite[Corollary 5.2]{RosOton2016:art} that $w \in L^{\infty}(\Omega)$, and thus $u \in L^{\infty}(\Omega)$.
\end{proof}

The second one gives an extension $g$ in $\Omega$, which is as smooth as possible inside~$\Omega$.

\begin{Lemma}\label{Lemma:PropExtensionBoundaryData}
Let $\Omega\subset \mathbb R^N$ be any bounded $C^1$ domain, and $g$ be as in \eqref{eq:AssumptionsBoundaryData}, with $\alpha\in(0,1)$.
Then, there exists a function $\overline{g} \in C^{0,\alpha}(\overline{\Omega}) \cap C^{\infty}(\Omega)$ such that
\begin{equation}\label{eq:BoundSecondDerivativesExtension}
\begin{array}{rcll}
\overline{g} &=& g \quad &\text{in } \mathbb{R}^N \setminus \Omega, \\
|D^2 \overline{g}| &\leq & Cd^{\alpha-2}            \quad &\text{in } \Omega,
\end{array}
\end{equation}
where $C$ depends only on $N$, $\alpha$, and $\Omega$.
\end{Lemma}

\begin{proof}
We consider the solution of $\Delta \overline{g} = 0$ in $\Omega$, $\overline{g} = g$ on $\partial\Omega$.
Since $g\in C^{0,\alpha}(\partial\Omega)$, and $\Omega$ is of class $C^1$, it follows from standard regularity theory that $\overline{g} \in C^{\infty}(\Omega)\cap C^{0,\alpha}(\overline{\Omega})$ and that\footnote{This can be shown by using that $\overline{g} \in C^{0,\alpha}(\overline{\Omega})$, standard elliptic regularity estimates, and the fact that the function $\overline{g} - \overline{g}(x_0)$ is harmonic in $\Omega$, for any choice of $x_0$.} $|D^2 \overline{g}| \leq C d^{\alpha-2}$ in $\Omega$.
\end{proof}

\begin{Lemma}\label{Lemma:PropExtensionBoundaryDataLip}
Let $\Omega\subset \mathbb R^N$ be any bounded Lipschitz domain.
Then, there exists $\alpha_0=\alpha_0(\Omega)$ such that the following holds.
Let $g$ be as in \eqref{eq:AssumptionsBoundaryData}, with $\alpha\in(0,1)$.

Then, there exists a function $\overline{g} \in C^{0,\beta}(\overline{\Omega}) \cap C^{\infty}(\Omega)$ such that
\begin{equation}\label{eq:BoundSecondDerivativesExtensionLip}
\begin{array}{rcll}
\overline{g} &=& g \quad &\text{in } \mathbb{R}^N \setminus \Omega, \\
|D^2 \overline{g}| &\leq & Cd^{\beta-2}            \quad &\text{in } \Omega,
\end{array}
\end{equation}
with $\beta := \min\{\alpha,\alpha_0\}$.
The constant $C > 0$ depends only on $N$, $\alpha$ and $\Omega$.
\end{Lemma}
\begin{proof} The proof is that of Lemma \ref{Lemma:PropExtensionBoundaryData}, recalling that when $\Omega$ is Lipschitz, then the harmonic extension of $g\in C^{0,\alpha}(\partial\Omega)$ satisfies $\overline{g} \in C^{0,\beta}(\overline{\Omega})$, with $\beta=\min\{\alpha,\alpha_0\}$.
\end{proof}

We next compute the operator $L$ evaluated on the extension $\overline g$ constructed above.

\begin{Lemma}\label{Lemma:BoundFracLapExtg}
(a) Let $\Omega\subset \mathbb R^N$ be any bounded $C^1$ domain, $s \in (0,1)$, $\alpha \in (0,2s)$, $L$ as in \eqref{eq:NonlocalNonHomogenousOperator}-\eqref{eq:KernelProp}-\eqref{eq:NonlocalHomogenousOperator}, $g$ as in \eqref{eq:AssumptionsBoundaryData}-\eqref{eq:GrowthCondition}, and $\overline{g}$ be given by Lemma \ref{Lemma:PropExtensionBoundaryData}.
Then,
\[
|L\overline{g}| \leq CC_0 d^{\alpha-2s} \quad \text{ in } \Omega,
\]
for some constant $C$ depending only on $N$, $s$, $\alpha$, $\Omega$, and the ellipticity constants.

(b) Let $\Omega\subset \mathbb R^N$ be any bounded Lipschitz domain,  $s \in (0,1)$, $\alpha \in (0,2s)$, $L$ as in \eqref{eq:NonlocalNonHomogenousOperator}-\eqref{eq:KernelProp}, and $g$ be as in \eqref{eq:AssumptionsBoundaryData}-\eqref{eq:GrowthCondition}, and $\overline{g}$ and $\alpha_0$ be given by Lemma \ref{Lemma:PropExtensionBoundaryDataLip}.
Then,
\[
|L\overline{g}| \leq CC_0 d^{\beta-2s} \quad \text{ in } \Omega,
\]
where $\beta=\min\{\alpha,\alpha_0\}$.
The constant $C$ depends only on $N$, $s$, $\alpha$, $\Omega$, and the ellipticity constants.
\end{Lemma}

\begin{proof} We prove (a) ---the same proof works for (b) replacing $\alpha$ by $\beta$.
As before, we may assume $C_0=1$.
Let $x_0 \in \Omega$, and define $\varrho := d(x_0)$.
Notice that we may assume $\varrho \in (0,\varrho_0)$ for some small $\varrho_0 > 0$ ---since $\overline{g} \in C^{\infty}(\Omega)$, the result is obvious if $d(x_0) \geq \varrho_0 > 0$.

Now, up to a positive multiplicative constant, we write
\[
\begin{aligned}
L \overline{g}(x_0) & = \frac{1}{2}\int_{B_{\varrho/2}} (\overline{g}(x_0 + y) + \overline{g}(x_0 - y) - 2\overline{g}(x_0)) K(y) dy \\
&\qquad + \frac{1}{2} \int_{\mathbb{R}^N \setminus B_{\varrho/2}} (\overline{g}(x_0 + y) + \overline{g}(x_0 - y) - 2\overline{g}(x_0)) K(y) dy =: \frac{1}{2} I_1 + \frac{1}{2} I_2.
\end{aligned}
\]
We notice that when $\alpha > s$, it is crucial to have also $\alpha < 2s$ so that the second integral above is finite.
Up to taking $\varrho_0 > 0$ smaller, the first integral can be estimated by using \eqref{eq:BoundSecondDerivativesExtension} as follows:
\[
|I_1| \leq C\int_{B_{\varrho/2}} \frac{|D^2 \overline{g}(x_0)||y|^2}{|y|^{N+2s}} dy \leq C \varrho^{\alpha - 2} \int_{B_{\varrho/2}}  |y|^{2-N-2s} dy = C \varrho^{\alpha - 2s}.
\]
To estimate the second integral, we pick a point $z_0 \in \partial\Omega$ such that $|x_0 - z_0| = \varrho$, and we consider
\[
\begin{aligned}
|I_2| & \leq C \int_{\mathbb{R}^N \setminus B_{\varrho/2}} \frac{|\overline{g}(x_0 + y) - \overline{g}(z_0)| + |\overline{g}(x_0 - y) - \overline{g}(z_0)| + 2|\overline{g}(x_0) - \overline{g}(z_0)|}{|y|^{N+2s}} dy \\
& \leq C \int_{\mathbb{R}^N \setminus B_{\varrho/2}} \frac{|x_0 - z_0 + y|^{\alpha} + |x_0 - z_0 - y|^{\alpha} + 2|x_0 - z_0|^{\alpha}}{|y|^{N+2s}} dy \\
& \leq C \int_{\mathbb{R}^N\setminus B_{\varrho/2}} \frac{(\varrho + |y|)^{\alpha} + \varrho^{\alpha}}{|y|^{N+2s}} dy + C \int_{\varrho/2}^{\infty} \frac{r^{\alpha}(\varrho/r + 1)^{\alpha} + \varrho^{\alpha}}{r^{1+2s}} dr \\
&\leq C \int_{\varrho/2}^{\infty} r^{\alpha - 1 - 2s} dr + C \varrho^{\alpha} \int_{\varrho/2}^{\infty} r^{-1-2s} dr = C \varrho^{\alpha - 2s}.
\end{aligned}
\]
In the second inequality we used that, since $\overline{g} \in C^{0,\alpha}(\overline{\Omega})$, if $x_0 \pm y \in \Omega$ then
\[
|\overline{g}(x_0 \pm y) - \overline{g}(z_0)| \leq C_0 |x_0 - z_0 \pm y|^{\alpha}
\]
while if $x_0 \pm y \in \mathbb{R}^N\setminus\Omega$, the same inequality holds by the assumption \eqref{eq:AssumptionsBoundaryData} on $g$.
The third one follows since $|x_0 - z_0 \pm y| \leq |x_0 - z_0| + |y| = \varrho + |y|$, while the last one since $\varrho/r \leq 2$.
Combining the estimates on $I_1$ and $I_2$, the lemma follows.
\end{proof}

We end the section with the following.
\begin{Lemma}
Let $s \in (0,1)$, $\alpha \in (0,s)$, $\nu \in \mathbb{S}^{N-1}$ and $L$ as in \eqref{eq:NonlocalHomogenousOperator}.
Then the function $\varphi_{\nu}^\alpha(x):= (x\cdot\nu)^{\alpha}_+$ satisfies
\[
-L \varphi_{\nu}^\alpha(x) > 0 \quad \text{ in } \; \{x\cdot\nu > 0\}.
\]
\end{Lemma}
\begin{proof}
First, note that $\varphi_{\nu}(x) = u(x\cdot \nu)$, where $u(t) := (t_+)^{\alpha}$.
Consequently (cfr. with \cite[Lemma 2.1 and Lemma 2.3]{RosOtonSerra2016:art}), it follows that
\[
-L\varphi_{\nu}(x) = \left(c_s \int_{\mathbb{S}^{N-1}}|\vartheta_n|^{2s} K(\theta)d\theta \right) (-\Delta)^s_{\mathbb{R}}u(x\cdot\nu) \quad \text{ in } \{x\cdot\nu > 0\},
\]
where $c_s > 0$ is a suitable constant.
Therefore, it suffices to prove the result for the fractional Laplacian $(-\Delta)^s_{\mathbb{R}}$ in dimension $N=1$.

For this, notice that $u$ is homogeneous of degree $\alpha$, and hence $Lu$ is homogeneous of degree $\alpha - 2s$.
Thus, it is enough to show that
\[
(-\Delta)^s_{\mathbb{R}} u (x_0) > 0 \quad \text{for some } x_0>0.
\]
We consider the function $v_c(x) := (x+c)^s_+$, which satisfies $(-\Delta)^s_{\mathbb{R}} v_c(x) = 0$ for all $x > -c$, where $c > 0$ is a free parameter (cfr. with \cite[Lemma 2.2]{RosOtonSerra2016:art}).
Since $v_c \geq u$ in $\mathbb R$ for $c$ large enough, it easy to see that there is $x_0>0$ and $c > 0$ such that
\[
v_c(x_0) = u(x_0) \qquad \textrm{and} \qquad v_c \geq u \quad \text{in } \mathbb{R}.
\]
Consequently, it follows
\[
v_c(x_0) - \frac{v_c(x_0 + y) + v_c(x_0 - y)}{2} \leq u_c(x_0) - \frac{u_c(x_0 + y) + u_c(x_0 - y)}{2},  \quad y \in \mathbb{R}
\]
(with strict inequality in $\mathbb R_+\setminus\{x_0\}$),
and so $0 = (-\Delta)^s_{\mathbb{R}}v_c(x_0) < (-\Delta)^s_{\mathbb{R}}u(x_0)$.
\end{proof}
%

%
%
%
%
%
%
%
%
%
%
%
\section{Proof of the main results}\label{Section:C1Domains}
The main goal of this section is to prove Theorem \ref{Theorem:MainTheorem}.
For this, we will use the following.

\begin{Lemma}[\cite{Lieberman1985:art}]\label{lemma-Lieberman}
Let $\Omega\subset \mathbb R^N$ be any bounded $C^1$ domain.
Then, there exists a modulus of continuity $\omega$ and a function $\psi\in C^1(\overline\Omega)$ satisfying
\begin{equation}\label{eq:PropPsi}
\begin{aligned}
&C^{-1} d \leq \psi \leq C d \quad \text{ in } \Omega, \\
&|\nabla \psi(x) - \nabla \psi(y)| \leq \omega(|x-y|) \quad \text{ for all } x,y \in \Omega, \\
&|D^2\psi(x)| \leq \omega(d(x)) \,d^{-1}(x) \quad \text{ for all } x \in \Omega,
\end{aligned}
\end{equation}
where $d(x)={\rm dist}(x,\Omega^c)$ and $C > 0$ is a constant depending only on $\Omega$.
\end{Lemma}

In the case of $C^{1,\gamma}$ domains $\Omega$, it is easy to see that one can choose $\omega(r) = Cr^{\gamma}$ in \eqref{eq:PropPsi}; see \cite[Definition 2.1]{RosOtonSerra2017:art}.
For a proof in case of general $C^1$ domains, we refer to \cite[Lemma 1.1 and Theorem 2.1]{Lieberman1985:art}.

We next prove two technical lemmas ---in case that $\Omega$ is $C^{1,\gamma}$, they correspond to Lemmas 2.4 and 2.5 of \cite{RosOtonSerra2017:art}.

\begin{Lemma}\label{PreliminaryLemma1}
Let $\Omega$ be any $C^1$ domain and let $\psi$ and $\omega$ be defined as in \eqref{eq:PropPsi}. Then, for each $x_0 \in \Omega$, it holds
\[
\left|\psi(x_0+y) - \big( \psi(x_0) + \nabla \psi(x_0) \cdot y \big)_+ \right| \leq C \omega(|y|)\,|y|, \quad y \in \mathbb{R}^N,
\]
where $C > 0$ depends only on $\Omega$.
\end{Lemma}

\begin{proof}
Since $\psi \in C^1(\overline{\Omega})$, there is an extension $\widetilde{\psi} \in C^1(\mathbb{R}^N)$ with $\widetilde{\psi} \leq 0$ in $\mathbb{R}^N \setminus \Omega$ and $\widetilde{\psi}|_{\Omega} = \psi$, preserving also the modulus of continuity $\omega$ of $\psi$ (up to a multiplicative constant).
Thus, if $x_0 \in \Omega$ we have
\[
\begin{aligned}
\bigg| \widetilde{\psi}(x) - \widetilde{\psi}(x_0) - \nabla \widetilde{\psi}(x_0) \cdot & (x-x_0)\bigg| = \left| \left( \nabla \widetilde{\psi}(\lambda x + (1-\lambda)x_0) - \nabla\widetilde{\psi}(x_0) \right)\cdot(x - x_0)  \right| \\
& \leq \left| \nabla \widetilde{\psi}(\lambda x + (1-\lambda)x_0) - \nabla\widetilde{\psi}(x_0) \right| \left| x - x_0 \right| \\
& \leq C \omega(\lambda|x-x_0|)|x-x_0| \leq C\omega(|x-x_0|)|x-x_0|,
\end{aligned}
\]
for all $x \in \mathbb{R}^N$, since $\lambda \in (0,1)$ and $\omega$ is increasing.
Now, using that $\widetilde{\psi}(x_0) = \psi(x_0)$, $\nabla \widetilde{\psi}(x_0) = \nabla \psi(x_0)$, $(\widetilde{\psi})_+ = \psi$, and that $|a_+-b_+| \leq |a-b|$, we obtain
\[
\left| \psi(x) - \big( \psi(x_0) + \nabla \psi(x_0) \cdot (x-x_0) \big)_+\right| \leq \omega(|x-x_0|)|x-x_0|,
\]
for all $x \in \mathbb{R}^N$, and the thesis follows.
\end{proof}
The next lemma is similar to \cite[Lemma 2.5]{RosOtonSerra2017:art}.
\begin{Lemma}\label{PreliminaryLemma2}
Let $\Omega$ be any $C^1$ domain, $x_0 \in \Omega$, $\varrho = d(x_0)/2$ and let $\omega$ be a modulus of continuity.
Then, there exists a modulus of continuity $\widetilde{\omega}$ such that, if $\delta > -1$ and $\beta \not= \delta$, then
\[
\int_{B_1 \setminus B_{\varrho/2}} d^{\delta}(x_0 + y) \frac{\omega(|y|)dy}{|y|^{N+\beta}} \leq C \left(1 + \widetilde{\omega}(\varrho)\varrho^{\delta - \beta}\right),
\]
for some constant $C > 0$ depending only on $\delta$, $\beta$, $\Omega$ and $\omega$.
\end{Lemma}
\begin{proof} Let us take $x_0 = 0$ (this can always be done up to a translation of the coordinate system), define $\varrho = d(0)/2$ and take $\kappa_{\ast} > 0$ such that the level sets $\{d = t\}$ are $C^1$ for all $t \in (0,\kappa_{\ast}]$ (this $\kappa_{\ast}$ exists since $\Omega \in C^1$). Without loss of generality, we can assume $\kappa_{\ast} > 2\varrho$ (i.e. $\varrho > 0$ small). Notice that if $\varrho \geq \varrho_0 > 0$ the inequality in our statement is just
\[
\int_{B_1 \setminus B_{\varrho/2}} d^{\delta}(x_0 + y) \frac{\omega(|y|)dy}{|y|^{N+\beta}} \leq C,
\]
and it is immediately verified. So, from now on, we will assume $0 < \varrho < \varrho_0$, for some small $\varrho_0$. First of all, we have
\begin{equation}\label{eq:IntEstimatedlarge}
\int_{(B_1 \setminus B_{\varrho/2})\cap \{d \geq \kappa_{\ast}\} } d^{\delta}(y) \frac{\omega(|y|)dy}{|y|^{N+\beta}} \leq C,
\end{equation}
where $C > 0$ depends only on $\delta$, $\beta$, $\Omega$ and $\omega$, thanks to the choice $\kappa_{\ast} > 2\varrho$.

Now we fix $M > M_0$ such that $2^{-M} \leq \varrho \leq 2^{-M+1}$ ($M_0$ is large and depends on $\varrho_0 > 0$) and, using the coarea formula, we obtain
\[
\begin{aligned}
\int_{(B_1 \setminus B_{\varrho/2})\cap \{d < \kappa_{\ast}\} } & d^{\delta}(y) \frac{\omega(|y|)}{|y|^{N+\beta}}dy \leq \sum_{k=0}^M \int_{(B_{2^{-k}} \setminus B_{2^{-k-1}})\cap \{d < \kappa_{\ast}\} } d^{\delta}(y) \frac{\omega(|y|)}{|y|^{N+\beta}}dy \\
& \leq \sum_{k=0}^M \frac{\omega(2^{-k})}{2^{-(N+\beta)k}} \int_{(B_{2^{-k}} \setminus B_{2^{-k-1}})\cap \{d < c2^{-k}\} } d^{\delta}(y) |\nabla d(y)| dy \\
& = \sum_{k=0}^M \frac{\omega(2^{-k})}{2^{-(N+\beta)k}} \int_0^{c2^{-k}} t^{\delta} dt \int_{(B_{2^{-k}} \setminus B_{2^{-k-1}})\cap \{d = t\} } d\mathcal{H}^{N-1}(y),
\end{aligned}
\]
for some $c > 0$ depending only on $\Omega$. Now, since $\Omega \in C^1$, we have for each $t \in (0,\kappa_{\ast})$
\[
\mathcal{H}^{N-1} \big( (B_{2^{-k}} \setminus B_{2^{-k-1}})\cap \{d = t\} \big) \leq C 2^{-(N-1)k},
\]
for some constant $C > 0$ depending only on $\Omega$. Consequently, it follows that
\[
\int_{(B_1 \setminus B_{\varrho/2})\cap \{d < \kappa_{\ast}\} } d^{\delta}(y) \frac{\omega(|y|)dy}{|y|^{N+\beta}} \leq C \sum_{k=0}^M \omega(2^{-k}) 2^{(\beta-\delta)k}.
\]
Thus if $\beta < \delta$, it is immediate to see that the above sum is bounded independently of $M$ (i.e. $\varrho$).
Keeping in mind that $2^{-M} \leq \varrho \leq 2^{-M+1}$, when $\beta > \delta$, it is enough to prove the existence of a modulus of continuity $\widetilde{\omega}$ such that
\begin{equation}\label{eq:BoundSumModCont}
\sum_{k=0}^M \omega(2^{-k}) 2^{(\beta-\delta)k} \leq  \widetilde{w}(2^{-M})\,2^{(\beta - \delta)M}.
\end{equation}
Finally, thanks to the Stolz-Cesaro theorem (l'Hopital rule for sequences), we obtain
\[
\begin{aligned}
\lim_{M\to\infty} \frac{\sum_{k=0}^M \omega(2^{-k}) 2^{(\beta-\delta)k}}{2^{(\beta-\delta)M}} &= \lim_{M\to\infty} \frac{\sum_{k=0}^{M+1} \omega(2^{-k}) 2^{(\beta-\delta)k} - \sum_{k=0}^M \omega(2^{-k}) 2^{(\beta-\delta)k}}{2^{(\beta-\delta)(M+1)} - 2^{(\beta-\delta)M}} \\
& = \lim_{M\to\infty} \frac{\omega(2^{-M+1}) 2^{(\beta-\delta)(M+1)} - \omega(2^{-M+1}) 2^{(\beta-\delta)M}}{2^{(\beta-\delta)M}} = 0,
\end{aligned}
\]
and recalling that $M > M_0$ for some large $M_0$, \eqref{eq:BoundSumModCont} follows, for some modulus of continuity~$\widetilde{\omega}$. Combining \eqref{eq:IntEstimatedlarge} and \eqref{eq:BoundSumModCont} we complete the proof of our statement.
\end{proof}

We next prove that the function $\psi^\alpha$ is a supersolution near $\partial\Omega$.

\begin{Lemma}\label{Lemma:SupersolutionPsia}
Let $s \in (0,1)$, $\alpha \in (0,s)$ and $L$ as in \eqref{eq:NonlocalNonHomogenousOperator}-\eqref{eq:KernelProp}-\eqref{eq:NonlocalHomogenousOperator}.
Then for any $C^1$ domain $\Omega$ and for any function $\psi$ satisfying \eqref{eq:PropPsi}, there is $\varrho_0 > 0$ such that:
\[
-L (\psi^{\alpha}) \geq c_0 d^{\alpha - 2s}>0 \quad \text{ in } \{0 < d < \varrho_0\},
\]
for some constant $c_0 > 0$ depending only on $N$, $s$, $\alpha$, $\Omega$, $\omega$ and the ellipticity constants.
\end{Lemma}
\begin{proof} Let $\alpha \in (0,s)$, $x_0 \in \Omega$ and $\varrho := d(x_0)$. We assume $\varrho \in (0,\varrho_0)$, for some $\varrho_0 > 0$ small which will be chosen later, and we consider the function
\[
l(x) := \big( \psi(x_0) + \nabla \psi(x_0) \cdot (x-x_0) \big)_+,
\]
satisfying $l(x_0) = \psi(x_0)$ and $\nabla \psi(x_0) = \nabla l(x_0)$. Notice that we can also assume $l > 0$ in $B_{\varrho/2}(x_0)$ and so, we obtain
\begin{equation}\label{eq:DeltaSoflaa}
-L(l^{\alpha})(x_0) = k l^{\alpha-2s}(x_0) = k \psi^{\alpha-2s}(x_0) \geq C \varrho^{\alpha - 2s},
\end{equation}
where we have used the assumptions on $\psi$ and set $k := l^{2s-\alpha}(x_0) [-L(l^{\alpha})(x_0)] = -L(l^{\alpha})(1) > 0$ (thanks to the homogeneity of $l^{\alpha}$ and \cite[Lemma 2.3]{RosOtonSerra2016:art}).

Now, from Lemma \ref{PreliminaryLemma1} we know that
\[
|\psi - l|(x_0 + y) \leq \omega(|y|)|y|, \quad y \in \mathbb{R}^N,
\]
and so, since $|a^{\alpha} - b^{\alpha}| \leq C |a-b|(a^{\alpha-1} + b^{\alpha -1})$ for all $a,b \geq 0$, it follows
\begin{equation}\label{eq:EstDiffPsiaLaB1Br2}
|\psi^{\alpha} - l^{\alpha}|(x_0 + y) \leq C (\psi^{\alpha-1} + l^{\alpha-1})(x_0 + y) \omega(|y|)|y|, \quad y \in \mathbb{R}^N,
\end{equation}
for some $C > 0$ depending only on $\Omega$ and $\alpha$, where we have used the first inequalities in \eqref{eq:PropPsi}. Furthermore, thanks to the properties of $\psi$, we have
\begin{equation}\label{eq:EstDiffSecondDerPsiaLaBr2}
|D^2(\psi^{\alpha} - l^{\alpha})| \leq C \omega(\varrho) \varrho^{\alpha - 2} \quad \text{ in } B_{\varrho/2}(x_0),
\end{equation}
for some new $C > 0$, which implies
\begin{equation}\label{eq:EstDiffPsiaLaBr2}
|\psi^{\alpha} - l^{\alpha}|(x_0 + y) \leq \|D^2(\psi^{\alpha} - l^{\alpha})\|_{L^{\infty}(B_{\varrho/2(x_0)})} |y|^2
\leq C \omega(\varrho)\varrho^{\alpha - 2}|y|^2
\end{equation}
for $y \in B_{\varrho/2}(x_0)$.
To check the validity of \eqref{eq:EstDiffSecondDerPsiaLaBr2}, we compute
\[
\begin{aligned}
(\psi^{\alpha} - l^{\alpha})_{x_jx_i}(x) &= \alpha(\alpha-1) \left[ \psi^{\alpha-2}(x) \psi_{x_i}(x) \psi_{x_j}(x) - l^{\alpha-2}(x)\psi_{x_i}(x_0)\psi_{x_j}(x_0)\right] \\
&\quad + \alpha \psi^{\alpha-1}(x) \psi_{x_ix_j}(x),
\end{aligned}
\]
and we notice that $|\psi^{\alpha-1}\psi_{x_ix_j}| \leq C \omega(\varrho)\varrho^{\alpha - 2}$ from \eqref{eq:PropPsi}.
On the other hand, we have
\[
\begin{aligned}
\psi^{\alpha-2}(x) \psi_{x_i}(x) \psi_{x_j}(x) -& l^{\alpha-2}(x)\psi_{x_i}(x_0)\psi_{x_j}(x_0) = \\
& = \psi^{\alpha-2}(x) \left[ \psi_{x_i}(x) \psi_{x_j}(x) - \psi_{x_i}(x_0)\psi_{x_j}(x_0)\right] +\\
&\qquad + \left[ \psi^{\alpha-2}(x) - l^{\alpha-2}(x)\right] \psi_{x_i}(x_0)\psi_{x_j}(x_0),
\end{aligned}
\]
and so, since $\nabla \psi$ is continuous up to $\partial \Omega$ with modulus of continuity $\omega(\cdot)$, it follows
\[
\big|\psi^{\alpha-2}(x) \left[ \psi_{x_i}(x) \psi_{x_j}(x) - \psi_{x_i}(x_0)\psi_{x_j}(x_0)\right]\big| \leq C \omega(\varrho) \varrho^{\alpha -2}, \quad x \in B_{\varrho/2}(x_0).
\]
Further, since $|D^2\psi(x)| \leq C \omega (d(x)) d^{-1}(x)$, we obtain
\[
\begin{aligned}
\big|\left[ \psi^{\alpha-2}(x) - l^{\alpha-2}(x)\right] \psi_{x_i}(x_0)\psi_{x_j}(x_0)\big| & \leq C |\psi(x) - l(x)| \big|\psi^{\alpha-3}(x) + l^{\alpha-3}(x)\big| \\
&\leq C |D^2\psi(x_0)||x-x_0|^2 \big|\psi^{\alpha-3}(x) + l^{\alpha-3}(x)\big|  \\
& \leq C \omega(\varrho) \varrho^{\alpha - 2}, \quad x \in B_{\varrho/2}(x_0),
\end{aligned}
\]
and so \eqref{eq:EstDiffSecondDerPsiaLaBr2} follows. Finally, since $\alpha \in (0,s)$ and $\psi = 0$ in $\mathbb{R}^N\setminus\Omega$,
\begin{equation}\label{eq:EstDiffPsiaLaRNB1}
|\psi^{\alpha} - l^{\alpha}|(x_0 + y) \leq C|y|^s, \quad y \in \mathbb{R}^N \setminus B_1.
\end{equation}
Consequently, if $\alpha + \gamma \not = 2s$, using \eqref{eq:DeltaSoflaa}, \eqref{eq:EstDiffPsiaLaB1Br2}, \eqref{eq:EstDiffPsiaLaBr2}, and \eqref{eq:EstDiffPsiaLaRNB1} it follows
\[
\begin{aligned}
-L (\psi^{\alpha}) (x_0) & = -L (l^{\alpha})(x_0) - L(\psi^{\alpha} - l^{\alpha})(x_0) \geq C \varrho^{\alpha - 2s}  - L(\psi^{\alpha} - l^{\alpha})(x_0) \\
& =  C \varrho^{\alpha - 2s} - \int_{\mathbb{R}^N} (\psi^{\alpha} - l^{\alpha})(x_0 + y) \frac{a(y/|y|) }{|y|^{N+2s}}dy \\
& \geq C \varrho^{\alpha - 2s} - C \int_{\mathbb{R}^N} |\psi^{\alpha} - l^{\alpha}|(x_0 + y) \frac{dy}{|y|^{N+2s}} \\
& \geq C \varrho^{\alpha - 2s} - C \omega(\varrho) \varrho^{\alpha - 2} \int_{B_{\varrho/2}} \frac{dy}{|y|^{N + 2s -2}} - C\int_{\mathbb{R}^N \setminus B_1} \frac{dy}{|y|^{N+s}} \\
& \quad - C \int_{B_1 \setminus B_{\varrho/2}} (d^{\alpha-1} + l^{\alpha-1})(x_0 + y) \frac{\omega(|y|)}{|y|^{N+2s-1}} dy \\
& \geq C \varrho^{\alpha - 2s} - C \omega(\varrho)\varrho^{\alpha - 2s} - C - C \left(1 + \widetilde{w}(\varrho)\varrho^{\alpha -2s} \right) \geq C \varrho^{\alpha - 2s},
\end{aligned}
\]
for some new constant $C > 0$ and all $0 < \varrho < \varrho_0$, where $\varrho_0 > 0$ depends only on $N$, $s$, $\alpha$, $\Omega$, $\omega$ and the ellipticity constants.
Notice that we have applied  Lemma \ref{PreliminaryLemma2} twice (once to $d(\cdot)$, once to $l(\cdot)$).
\end{proof}
\begin{proof}[Proof of Theorem \ref{Theorem:MainTheorem}]
Dividing $g$ and $u$ by a constant if necessary, we may assume $C_0=1$.
Thanks to Lemma \ref{Lemma:LInfinityBound}, we have  $\|u\|_{L^{\infty}(\Omega)} \leq C$.
On the other hand, let $\overline{g}$ denotes the extension of $g$ given by Lemma \ref{Lemma:PropExtensionBoundaryData}.
Then, the function
\[v = u - \overline{g}\]
 solves \eqref{eq:SecondNonlocalDirichletProblem} with $f := L \overline{g}$.
Moreover, thanks to Lemma \ref{Lemma:BoundFracLapExtg} we have $|f| \leq C d^{\alpha-2s}$ in $\Omega$.

\vspace{2mm}

\noindent \emph{Step 1.} We claim that
\begin{equation}\label{eq:CrucialBoundMainTheorem}
|v| \leq C d^{\alpha} \quad \textrm{in}\quad \Omega.
\end{equation}
To prove this, we consider the function
\[
\varphi(x) := M \psi^{\alpha}(x), \quad x \in \mathbb{R}^N,
\]
where $\psi$ is given by Lemma \ref{lemma-Lieberman}.
Thanks to Lemma \ref{Lemma:SupersolutionPsia}, for some $\varrho_0 > 0$ we have
\[
-L \varphi \geq M C d^{\alpha - 2s} \quad \text{ in } \Omega_{\varrho_0} := \big\{ x \in \Omega: 0 < d(x) < \varrho_0 \big\},
\]
for some constant $C > 0$ depending only on $N$, $s$, $\alpha$, $\Omega$, and the ellipticity constants.
We now compare $\varphi$ with $v$:

\begin{itemize}

\item We clearly have $\varphi = v = 0$ in $\mathbb{R}^N \setminus \Omega$, for all $M > 0$.

\item Since $\|v\|_{L^{\infty}(\Omega)} \leq C$ and $\psi^\alpha>0$ in $\Omega$, taking $M > 0$ large enough we have
\[
\varphi \geq v \quad \text{ in } \Omega \setminus \Omega_{\varrho_0}.
\]

\item Thanks to Lemma \ref{Lemma:SupersolutionPsia}, we may choose $M > 0$ large enough such that
\[
-L \varphi \geq -Lv \quad \text{ in } \Omega_{\varrho_0}.
\]

\end{itemize}
Consequently, taking $M$ large enough (depending only on $n$, $s$, $\alpha$, $\Omega$, $\lambda$, and $\Lambda$), it follows from the comparison principle that $v \leq \varphi$ in $\Omega$.
Repeating the above argument with $-v$, \eqref{eq:CrucialBoundMainTheorem} follows.

\vspace{2mm}

\noindent \emph{Step 2.} We next claim that
\begin{equation}\label{eq:AlphaHolderEstimateFinal}
[v]_{C^{0,\alpha}(B_r(x_0))} \leq {C},
\end{equation}
for any ball $B_r(x_0) \subset \Omega$ with $d(x_0) = 2r$, and some constant ${C}$ independent of $x_0 \in \Omega$ and $r > 0$.

To do this, we recall\footnote{This follows for example from \cite[Theorem 1.1]{RosOtonSerra2016Bis:art} and \cite[Theorem 5.1]{CS11}.} that if $w$ is a solution to $-L w = f$ in $B_2$ then
\begin{equation}\label{eq:AlphaHolderInteriorEstimate}
[w]_{C^{0,\alpha}(B_1)} \leq C \left( \|f\|_{L^\infty(B_2)} + \int_{\mathbb{R}^N} \frac{|w(x)|}{1 + |x|^{N+2s}} \,dx \right).
\end{equation}
Now, for any $x_0 \in \Omega$ and $r := d(x_0)/2$, we take $w(x) := v(x_0 + rx)$,
which satisfies
\[
-L w(x) = f_r(x) := r^{2s} f(x_0 + rx) \leq C r^{2s} d^{\alpha-2s}(x_0 + rx) \leq C r^{\alpha} \quad \text{ in } B_2,
\]
since $x_0 + rx \in B_{2r}(x_0)$.
On the other hand, since $|v| \leq C d^{\alpha}$ in $\Omega$ (thanks to Step 1), we have
\[
|w(x)| = |v(x_0 + rx)| \leq C d^{\alpha}(x_0 + rx) \leq C r^{\alpha}(1 + |x|^{\alpha}) \quad \text{ in } \mathbb{R}^N,
\]
where $C > 0$ is a new constant independent of $x_0 \in \Omega$ and $r > 0$, and so
\[
\int_{\mathbb{R}^N} \frac{|w(x)|}{1 + |x|^{N+2s}}\,dx \leq C r^{\alpha} \int_{\mathbb{R}^N} \frac{1 + |x|^{\alpha}}{1 + |x|^{N+2s}}\,dx \leq C r^{\alpha},
\]
since $\alpha < 2s$.
Above, we have used the fact that
\[
d(x_0 + rx) = \inf_{y \in \partial\Omega} |y - (x_0 + rx)| \leq \inf_{y \in \partial\Omega} |y - x_0| + r|x| = d(x_0) + r|x| =  r(2 + |x|),
\]
whenever $x_0 + rx \in \Omega.$ Consequently, applying \eqref{eq:AlphaHolderInteriorEstimate}, we obtain
\[
[w]_{C^{0,\alpha}(B_1)} \leq C r^{\alpha},
\]
from which \eqref{eq:AlphaHolderEstimateFinal} immediately follows.

\vspace{2mm}

\noindent \emph{Step 3.} We can now finish the proof.
Indeed, take $x,y \in \overline{\Omega}$, with $r = |x-y|$ and $\varrho = \min\{d(x),d(y)\}$.
There are two possibilities:

\begin{itemize}

\item If $\varrho \leq 2r$, assuming for instance $\varrho = d(y)$ and recalling \eqref{eq:CrucialBoundMainTheorem}, we have
\[
|v(x) - v(y)| \leq |v(x)| + |v(y)| \leq C \left[ d^{\alpha}(x) + d^{\alpha}(y) \right] \leq C \left[ (\varrho + r)^{\alpha} + \varrho^{\alpha} \right] \leq \widetilde{C} r^{\alpha},
\]
for some constant $\widetilde{C} > 0$ independent of $x,y \in \Omega$ and $r > 0$.

\item If $\varrho > 2r$ and $\varrho = d(y)$, it follows that $B_{2r}(x) \subset \Omega$, and thus thanks to \eqref{eq:AlphaHolderEstimateFinal}
\[
|v(x) - v(y)| = r^{\alpha} \frac{|v(x) - v(y)|}{|x-y|^{\alpha}} \leq r^{\alpha} \sup_{y \in B_r(x)} \frac{|v(x) - v(y)|}{|x-y|^{\alpha}} \leq r^{\alpha} [v]_{C^{0,\alpha}(B_r(x))} \leq \overline{C} r^{\alpha}.
\]

\end{itemize}

Putting together the last two inequalities we find
\[
|v(x) - v(y)| \leq C |x - y|^{\alpha} \quad\textrm{ for all }x,y \in \overline{\Omega}.
\]
This implies that $v \in C^{0,\alpha}(\overline{\Omega})$ and, therefore, that $u \in C^{0,\alpha}(\overline{\Omega})$.
\end{proof}

\begin{proof}[Proof of Proposition \ref{Proposition:MainProp}]
Let $\Omega := \mathbb{R}_+^2 := \{(x_1,x_2) \in \mathbb{R}^2: x_2 > 0\}$ and $u$ solve
\begin{equation}\label{eq:DirichletProblemNonlocalCase}
\left\{\begin{array}{rcll}
(-\Delta)^s u &=& 0 \quad &\text{in } \Omega, \\
u &= & g            \quad &\text{in } \mathbb{R}^2 \setminus \Omega,
\end{array}\right.
\end{equation}
with  $g(x) := \min\{|x|^s,\,1\}$.
The solution to \eqref{eq:DirichletProblemNonlocalCase} is given by the Poisson kernel
\[
u(x_1,x_2) = c_sx_2^s \int_{\mathbb{R}} \int_{-\infty}^0 \frac{g(z_1,z_2)}{|z_2|^s [(x_1-z_1)^2 + (x_2-z_2)^2]} dz_1dz_2;
\]
see \cite{BlumGetoorRay1961:art}.
Notice that, since $g\geq0$, we clearly have
\[u(x_1,x_2) \geq  c_sx_2^s\int_S \frac{g(z_1,z_2)}{|z_2|^s [z_1^2 + (t-z_2)^2]} dz_1dz_2, \]
where $S := \{x_2 \leq -|x_1|\} \cap B_1$.
Thus, in order to prove that $u \not\in C^{0,s}(\overline{\Omega})$, it is enough to show that the last integral
is unbounded for $x_1 = 0$ and $x_2 > 0$ small.

For this, we set $t = x_2$,
and we consider
\[
\begin{aligned}
\int_S \frac{g(z_1,z_2)}{|z_2|^s [z_1^2 + (t-z_2)^2]} dz_1dz_2 & = \int_S \frac{(z_1^2 + z_2^2)^{\frac{s}{2}}}{|z_2|^s [z_1^2 + (t-z_2)^2]} dz_1dz_2 \\
& = \int_0^1 \int_{\frac{5}{4}\pi}^{\frac{7}{4}\pi} \frac{r^{s+1}}{r^s|\sin \theta|^s (r^2 + 2r|\sin\theta| t + t^2)} d\theta dr \\
& =  \int_{\frac{5}{4}\pi}^{\frac{7}{4}\pi} |\sin \theta|^{-s} \int_0^1 \frac{r}{r^2 + 2r|\sin\theta| t + t^2} dr d\theta,
\end{aligned}
\]
where we have passed to polar coordinates $z_1 = r\cos\theta$, $z_2 = r \sin \theta$.
Since $|\sin \theta| \in [\sqrt{2}/2,1]$, it follows that
\[
\begin{aligned}
\int_0^1 \frac{r}{r^2 + 2r|\sin\theta| t + t^2} dr &\geq \int_0^1 \frac{r}{r^2 + 2rt + t^2} dr = \frac{1}{2} \int_0^1 \frac{2(r+t)}{(r + t)^2} dr - t \int_0^1 \frac{dr}{(r+t)^2} \\
& = \ln \left(\frac{1}{t}\right) + \ln (1+t) - \frac{1}{1+t} = \ln \left(\frac{1}{t}\right) + O(1)
\end{aligned}
\]
as $t \to 0^+$.
Consequently, setting $\kappa_s := \int_{\frac{5}{4}\pi}^{\frac{7}{4}\pi} |\sin \theta|^{-s}$, we obtain
\[
u(0,t) \geq c_st^s \int_S \frac{g(z_1,z_2)}{|z_2|^s [z_1^2 + (t-z_2)^2]} dz_1dz_2 \geq c_s\kappa_s t^s\ln \left(\frac{1}{t}\right) + O(t^s), \quad \text{as } t \to 0^+,
\]
which shows that $u\notin C^{0,s}(\overline{\Omega})$. 

Finally, notice that using such function $u$ one can actually construct a counterexample in a bounded domain, too.
\end{proof}

\begin{proof}[Proof of Proposition \ref{Proposition:MainProp2}]
As in the proof of Theorem \ref{Theorem:MainTheorem}, we have that $v = u - \overline{g}$ solves \eqref{eq:SecondNonlocalDirichletProblem}, with $|f| \leq Cd^{\alpha - 2s}$ in $\Omega$.
Setting $\theta:=\min\{\alpha-s,\gamma\}>0$, we can apply \cite[Theorem 1.2, Proposition 3.2]{RosOtonSerra2017:art} (cfr. with \cite[Remark 3.4]{RosOtonSerra2017:art}) to conclude that $v \in C^{0,s}(\overline{\Omega})$.
\end{proof}

%
%
%
%
%
%
%
%
%
%
%
\section{Lipschitz domains}\label{Section:LipschitzDomains}
We focus now on the case in which $\Omega$ is Lipschitz.
Again, the idea is to construct a suitable super-solution for problem \eqref{eq:SecondNonlocalDirichletProblem} exploiting that the r.h.s. explode as $d^{\beta-2s}$ near the boundary $\partial\Omega$, for some $\beta \in (0,1)$.
Since the domain is only Lipschitz and the operator $L$ is not homogeneous, the strategy followed in the above section cannot work in this more general framework.
The construction of a new barrier is the main difficulty here, and it works as follows.

For any fixed direction $e \in \mathbb{S}^{N-1}$ and $\eta > 0$, we consider the function
\[
\psi(x) := e \cdot x + \eta |x| \left( 1 - \frac{(e \cdot x)^2}{|x|^2}\right), \quad x \in \mathbb{R}^N,
\]
and, for any $\beta \in (0,1)$, we define
\begin{equation}\label{eq:SupersolutionLipschitz}
\Phi_{\beta} := (\psi_+)^{\beta}.
\end{equation}
Notice that $\Phi_{\beta} \in C^{0,\beta}(\mathbb{R}^N)$ and it is positive in the cone
\begin{equation}\label{eq:ConeDefinition}
\mathcal{C}_{-\eta} := \left\{ x \in \mathbb{R}^N : \frac{e \cdot x}{|x|} > - \eta \left( 1 - \frac{(e \cdot x)^2}{|x|^2} \right) \right\},
\end{equation}
while zero otherwise. We begin with the following lemma (cfr. also with \cite{DiPierroSoaveValdinoci:art,RosOtonSerra2018:art}).
\begin{Lemma}\label{Lemma:SuperLipschitz}
Let $s \in (0,1)$, $e \in \mathbb{S}^{N-1}$ and $L$ as in \eqref{eq:NonlocalNonHomogenousOperator}-\eqref{eq:KernelProp}. Then for all $\eta > 0$, there exists $\beta_0 \in (0,1)$ such that the function $\Phi_{\beta}$ defined in \eqref{eq:SupersolutionLipschitz} satisfies
\begin{equation}\label{eq:SupersolutionLipschitzProb}
\left\{\begin{array}{rcll}
-L \Phi_{\beta} &\geq & c_0 d^{\beta-2s}>0 \quad &\text{in }  \mathcal{C}_{-\eta} \\
\Phi_{\beta} & =& 0             \quad &\text{in } \mathbb{R}^N \setminus \mathcal{C}_{-\eta},
\end{array}\right.
\end{equation}
for all $\beta \in (0,\beta_0]$ and some $c_0 > 0$.
The constants $\beta_0$ and $c_0$ depend only on $n$, $s$, $\lambda$, $\Lambda$, and $\eta$.
\end{Lemma}
\begin{proof} As in \cite[Lemma 4.1]{RosOtonSerra2017:art}, using the homogeneity of $\Phi_{\beta}$ and $d^{\beta-2s}$, and the properties of the kernel $K$, it is enough to prove
\[
-L \Phi_{\beta} \geq C > 0 \quad \text{ in } e + \partial \mathcal{C}_{-\eta}.
\]
This is because $\{\lambda (e + \partial \mathcal{C}_{-\eta}) \}_{\lambda > 0} = \mathcal{C}_{-\eta}$.

Now, let us take $\varrho = \varrho(\eta) > 0$ so that
\[
0 < \varrho < \inf_{\substack{x \in \partial \mathcal{C}_{-\eta} \\ y \in e + \partial \mathcal{C}_{-\eta}}} |x-y| := \text{dist} \left(\partial \mathcal{C}_{-\eta}, e + \partial \mathcal{C}_{-\eta} \right),
\]
and notice that
\[|\nabla \psi| \leq C \quad \text{ and } \quad |D^2 \psi| \leq C|x|^{-1} \quad \text{in } \mathbb{R}^N \setminus \{0\},
\]
for some constant $C$ depending on $\eta$.
Consequently, it follows that
\begin{equation}\label{eq:C2boundsonPhialpha}
\|\Phi_{\beta}\|_{C^2(B_{\varrho}(x))} \leq C \quad \text{ for all } x \in e + \partial \mathcal{C}_{-\eta},
\end{equation}
for some new constant $C > 0$ independent of $x$.
On the other hand, we have
\begin{equation}\label{eq:AEconvergencePhiAlpha}
\Phi_{\beta} \to \Phi_0 := \chi_{\mathcal{C}_{-\eta}} \quad \text{ pointwise in } \mathbb{R}^N
\end{equation}
as $\beta \to 0^+$, and moreover
\begin{equation}\label{eq:CalphaboundPhialpha}
|\Phi_{\beta}(x) - \Phi_{\beta}(x-y)| \leq C(1 + |y|^{\beta_0}) \quad \text{ for all } x \in e + \mathbb{R}^N \setminus \mathcal{C}_{-\eta}, \; y \in \mathbb{R}^N,
\end{equation}
for all $\beta \in (0,\beta_0]$ and some constant $C$ independent of $x,y \in \mathbb{R}^N$ and independent of $\beta$.
Finally, let us define
\[
h_{\beta}(x,y) := \left( \Phi_{\beta}(x) - \frac{\Phi_{\beta}(x+y) + \Phi_{\beta}(x-y)}{2} \right) K(y),
\]
Thus, using \eqref{eq:C2boundsonPhialpha} and \eqref{eq:CalphaboundPhialpha}, we obtain
\[
|h_{\beta}(x,y)| \leq C |y|^{-N - 2(1-s)} \quad \text{ for all } x \in e + \partial \mathcal{C}_{-\eta}, \; y \in B_{\varrho},
\]
and
\[
|h_{\beta}(x,y)| \leq C \left( 1 + |y|^{- N - 2s + \beta_0} \right) \quad \text{ for all } x \in e + \partial \mathcal{C}_{-\eta}, \; y \in \mathbb{R}^N,
\]
for some constant $C$ independent of $x$ and $y$ (actually it depends only on $n$, $s$, $\lambda$, $\Lambda$, $\alpha$, and $\Omega$). Recalling \eqref{eq:AEconvergencePhiAlpha}, we can apply Lebesgue dominated convergence theorem to deduce
\begin{equation}\label{eq:UniformConvergenceOperator}
\begin{aligned}
-L \Phi_{\beta}(x) = \int_{B_{\varrho}} h_{\beta}(x,y) dy  +  \int_{\mathbb{R}^N \setminus B_{\varrho}} h_{\beta}(x,y) dy \to -L \Phi_0(x),
\end{aligned}
\end{equation}
as $\beta \to 0$, for every fixed $x \in e + \partial \mathcal{C}_{-\eta}$.
Actually, the above convergence is uniform w.r.t. $x \in e + \partial\mathcal{C}_{-\eta}$ and $L$.
Indeed, assume by contradiction that for any sequence $\beta_j \to 0^+$, there exist a sequence $\{x_j\}_j \in e + \partial \mathcal{C}_{-\eta}$ and a sequence of operators $\{L_j\}_j$ satisfying \eqref{eq:NonlocalNonHomogenousOperator}-\eqref{eq:KernelProp} such that
\begin{equation}\label{eq:ContradictionInequality}
|L_j \Phi_{\beta_j}(x_j) - L_j\Phi_0(x_j)| \geq \varepsilon
\end{equation}
for some $\varepsilon > 0$ and all $j \in \mathbb{N}$. Then, using again the bounds in \eqref{eq:C2boundsonPhialpha} and \eqref{eq:CalphaboundPhialpha}, we easily deduce that $|L_j \Phi_{\beta_j}(x_j)| \leq C$, for some constant $C$ independent of $j \in \mathbb{N}$ and so, up to passing to a subsequence, we obtain that $L_j\Phi_{\beta_j}(x_j)$ has a finite limit as $j \to +\infty$.
Further, using the pointwise convergence in \eqref{eq:AEconvergencePhiAlpha} and a standard diagonal procedure, we can extract a subsequence $\{\beta_{j_k}\}_k \subset \{\beta_j\}_j$ for which $|L_k \Phi_{\beta_{j_k}}(x_k) - L_k\Phi_0(x_k)| \leq \varepsilon/2$ for all $k \in \mathbb{N}$, obtaining a contradiction with \eqref{eq:ContradictionInequality}.

On the other hand, since $\Phi_0=\chi_{\mathcal C_{-\eta}}$, we obtain
\[
-L \Phi_0(x) = \int_{\mathbb{R}^N} \left( \Phi_0(x) - \Phi_0(y) \right) K(x-y) dy = \int_{\mathbb{R}^N \setminus \mathcal{C}_{-\eta}} K(x-y) dy.
\]
Consequently, writing $x = e + P$ with $P \in \partial \mathcal{C}_{-\eta}$ and noticing that $\mathbb{R}^N \setminus \mathcal{C}_{-\eta} \subset - P + \mathbb{R}^N \setminus \mathcal{C}_{-\eta}$, it follows
\[
\begin{aligned}
-L \Phi_0(x) &= \int_{- P + \mathbb{R}^N \setminus \mathcal{C}_{-\eta}} K(y-e) dy \geq \int_{\mathbb{R}^N \setminus \mathcal{C}_{-\eta}} K(y-e) dy \geq c > 0,
\end{aligned}
\]
for some $c > 0$ independent of $x \in e + \partial \mathcal{C}_{-\eta}$. Thus, recalling the uniform convergence in \eqref{eq:UniformConvergenceOperator}, we deduce the existence of a small $\beta_0 \in (0,1)$ such that
\[
-L\Phi_{\beta}(x) \geq c/2 > 0,
\]
for all $x \in e + \partial \mathcal{C}_{-\eta}$ and $\beta \in (0,\beta_0]$.
\end{proof}

\begin{proof}[Proof of Theorem \ref{Theorem:MainTheoremLipschitz}] Let $\beta_0>0$ be given by Lemma \ref{eq:SupersolutionLipschitzProb}, and assume without loss of generality that $\beta_0\leq\alpha_0$, where $\alpha_0$ is given by Lemma \ref{Lemma:PropExtensionBoundaryDataLip}.

Let  $v = u - \overline{g}$, where $\overline{g}$ is chosen as in Lemma \ref{Lemma:PropExtensionBoundaryDataLip}.
Thanks to Lemma \ref{Lemma:BoundFracLapExtg}(b), we have $|L\overline{g}| \leq C d^{\alpha-2s}$ in $\Omega$, and so
\begin{equation}\label{eq:BoundonLvLipschitz}
|L v| \leq C d^{\alpha-2s} \quad \text{ in } \Omega,
\end{equation}
since $v$ satisfies \eqref{eq:SecondNonlocalDirichletProblem} with $f = L\overline{g}$.
We want to prove that
\begin{equation}\label{eq:BoundDecayEstimateLipschitz}
|v(x)| \leq C|x - z_0|^{\alpha} = Cd^{\alpha}(x) \quad \text{ in } \Omega,
\end{equation}
where $z_0 \in \partial\Omega$ is the projection of $x$ on $\partial\Omega$.
To do this, since $\partial\Omega$ is Lipschitz we have that for any  point $z_0 \in \partial \Omega$, there are $r,\eta > 0$ such that
\[
B_{2r}(z_0) \cap \Omega \subseteq B_{2r}(z_0) \cap \left( z_0 + \mathcal{C}_{-\eta} \right),
\]
where $\mathcal{C}_{-\eta}$ is defined in \eqref{eq:ConeDefinition}. Moreover, we can choose $r > 0$ and $\eta > 0$ independently of $z_0 \in \Omega$ (i.e. $x \in \Omega$).

Now, we consider the truncation
\[
w := v \chi_{B_{2r}(z_0)},
\]
which satisfies $|Lw| \leq C d^{\alpha -2s}$ in $B_r(z_0) \cap \Omega$, thanks to \eqref{eq:BoundonLvLipschitz}.

On the other hand, we consider the function
\[
\varphi(x) := M \Phi_{\alpha}(x - z_0) \quad x \in \mathbb{R}^N,
\]
where $\Phi_{\alpha}$ is defined in \eqref{eq:SupersolutionLipschitz}, and $M>0$ is to be chosen.
Thanks to Lemma \ref{Lemma:SuperLipschitz},
\[
-L \varphi \geq Mc_0 d^{\alpha - 2s}>0 \quad \text{ in } B_r(z_0) \cap \left( z_0 + \mathcal{C}_{-\eta} \right).
\]
Now, choose $M > 0$ so that
\[
\varphi \geq w \quad \text{ in } \mathbb{R}^N \setminus (B_r(z_0) \cap \Omega)
\]
---which is possible since $w$ is bounded---, and so that
\[
-L \varphi \geq Mc_0d^{\alpha - 2s} \geq C d^{\alpha -2s} \geq -Lw \quad \text{ in } B_r(z_0) \cap \Omega.
\]
By the maximum principle, we deduce
\[
v=w \leq C \varphi \leq C|x - z_0|^{\beta} = C d^{\beta}(x) \quad \text{ in } B_r(z_0) \cap \Omega
\]
and, repeating the above argument with $w$ replaced by $-w$, \eqref{eq:BoundDecayEstimateLipschitz} follows.

To finish the proof, we can repeat the proof of Theorem \ref{Theorem:MainTheorem}, combing \eqref{eq:BoundDecayEstimateLipschitz} with the interior estimate \eqref{eq:AlphaHolderInteriorEstimate}.
\end{proof}

As a consequence, we find:

\begin{Corollary}\label{cor-existence}
Let $N$, $\Omega$, $s$, $L$, and $g$ be as in Theorem \ref{Theorem:MainTheoremLipschitz}.
Then, there exists a viscosity solution $u\in C(\overline\Omega)$ to \eqref{eq:NonlocalDirichletProblem}.
\end{Corollary}

\begin{proof}
Let $\Omega_\varepsilon\subset\subset \Omega$ be a sequence of smooth domains such that $\Omega_\varepsilon\to \Omega$ in the Hausdorff distance, and such that $\Omega_\varepsilon$ are Lipschitz sets (uniformly in $\varepsilon$).

Let $\bar g$ be a H\"older continuous extension of $g$ inside $\Omega$, and let $g_\varepsilon\in C^\infty_c(\mathbb R^N)$ be such that $g_\varepsilon\to \bar g$ uniformly in $\Omega$, $g_\varepsilon\to g$ a.e. in $\mathbb R^N\setminus \Omega$, and such that \eqref{eq:AssumptionsBoundaryData} holds uniformly in $\varepsilon$.

Then, by \cite[Theorem 5.6]{Mou2017:art} there exists a viscosity solution $u_\varepsilon$ to $Lu_\varepsilon=0$ in $\Omega_\varepsilon$, $u_\varepsilon=g_\varepsilon$ in $\Omega_\varepsilon^c$.
By Theorem  \ref{Theorem:MainTheoremLipschitz}, we have a uniform bound $\|u_\varepsilon\|_{C^{0,\alpha}(\overline\Omega)}\leq C$, with $\alpha>0$.
By Arzel\`a-Ascoli theorem, we have $u_\varepsilon\to u$ uniformly in $\overline\Omega$, up to subsequence, where $u\in C^{0,\alpha}(\overline \Omega)$.
Moreover, $u_\varepsilon\to u$ almost everywhere  in $\mathbb R^N\setminus \Omega$, where $u:=g$ outside $\Omega$.
Then, by stability of viscosity solutions (see, e.g. \cite[Lemma 4.5]{CS09}), we have that $u$ is a viscosity solution of \eqref{eq:NonlocalDirichletProblem}.
\end{proof}


\begin{thebibliography}{99}
%

\bibitem{BarlesImbert2008:art}
{\sc G. Barles, C. Imbert.} \emph{Second-order elliptic integro-differential equations: Viscosity solutions' theory revisited}, Ann. Inst. H. Poincar\'e Anal. Non Lin\'eaire \textbf{25} (2008), 567--585.

\bibitem{BlumGetoorRay1961:art}
{\sc R. M. Blumenthal, R. K. Getoor, D. B. Ray.} \emph{On the distribution of first hits for the symmetric
stable processes}, Trans. Amer. Math. Soc. \textbf{99} (1961), 540--554.

\bibitem{Bogdan1997:art}
{\sc K. Bogdan.} \emph{The boundary Harnack principle for the fractional Laplacian}, Studia Math. \textbf{123} (1997), 43--80.
%

\bibitem{BogdanGrzywnyRyznar2015:art}
{\sc K. Bogdan, T. Grzywny, M. Ryznar.} \emph{Barriers, exit time and survival probability for unimodal L\'evy processes}, Probab. Theory Related Fields \textbf{162} (2015), 155--198.
%

\bibitem{BogdanKumagaiQuitalo2015:art}
{\sc K. Bogdan, T. Kumagai, M. Kwa\'snicki.} \emph{Boundary Harnack inequality for Markov processes with jumps}, Trans. Amer. Math. Soc. \textbf{367} (2015), 477--517.
%

\bibitem{CS09}
{\sc L. A. Caffarelli, L. Silvestre.} \emph{Regularity theory for fully nonlinear integro-differential
equations}, Comm. Pure Appl. Math. \textbf{62} (2009), 597--638.
%

\bibitem{CS11}
{\sc L. Caffarelli, L. Silvestre.} \emph{The Evans-Krylov theorem for nonlocal fully nonlinear equations}, Ann. Math. \textbf{174} (2011), 1163--1187.
%

\bibitem{DiPierroSoaveValdinoci:art}
{\sc S. Dipierro, N. Soave, E. Valdinoci}. \emph{On fractional elliptic equations in Lipschitz sets and epigraphs: regularity, monotonicity and rigidity results}, Math. Ann. \textbf{369} (2017), 1283--1326.

\bibitem{FelKassVoigt2015:art}
{\sc M. Felsinger, M. Kassmann, P. Voigt.} \emph{The Dirichlet problem for nonlocal operators}, Math.~Z. \textbf{279} (2015), 779--809.

\bibitem{GilbargTrudinger2001:book}
{\sc D. Gilbarg, N. S. Trudinger.} ``Elliptic Partial Differential Equations of Second Order'', Springer-Verlag Berlin Heidelberg, 1998.

\bibitem{Grubb2014}
{\sc G. Grubb.} \emph{Local and nonlocal boundary conditions for $\mu$-transmission and fractional elliptic pseudodifferential operators}, Anal. PDE \textbf{7} (2014), 1649--1682.


%


\bibitem{Kenig:book}
{\sc C. E. Kenig.} ``Harmonic Analysis Techniques for Second Order Elliptic Boundary Value Problems'', CBMS of the AMS, Regional conference series in mathematics \textbf{83}, 1994.

\bibitem{Lieberman1985:art}
{\sc G. M. Lieberman.} \emph{Regularized distance and its applications}, Pacific J. Math. \textbf{117} (1985), 329--352.

\bibitem{Mou2017:art}
{\sc C. Mou.} \emph{Perron's method for nonlocal fully nonlinear equations}, Anal. PDE \textbf{10} (2017), 1227--1254.
%

\bibitem{RosOton2016:art}
{\sc X. Ros-Oton.} \emph{Nonlocal elliptic equations in bounded domains: a survey},
Publ. Mat. \textbf{60} (2016), 3--26.

%

\bibitem{RosOtonSerra2016:art}
{\sc X. Ros-Oton, J. Serra.} \emph{Boundary regularity for fully nonlinear integro-differential equations},
Duke Math. J. \textbf{165} (2016), 2079--2154.

\bibitem{RosOtonSerra2016Bis:art}
{\sc X. Ros-Oton, J. Serra.} \emph{Regularity theory for general stable operators}, J. Differential Equations \textbf{260} (2016), 8675--8715.

\bibitem{RosOtonSerra2017:art}
{\sc X. Ros-Oton, J. Serra.} \emph{Boundary regularity estimates for nonlocal elliptic equations in $C^1$ and $C^{1,\alpha}$ domains}, Ann. Mat. Pura Appl. \textbf{196} (2017), 1637--1668.

\bibitem{RosOtonSerra2018:art}
{\sc X. Ros-Oton, J. Serra.} \emph{The boundary Harnack principle for nonlocal elliptic operators in non-divergence form}, Potential Anal. \textbf{51} 51 (2019), 315-331.

\end{thebibliography}
\end{document}